\documentclass{amsart}
\usepackage{amssymb,latexsym}
\theoremstyle{plain}
\newtheorem{theorem}{Theorem}

\newtheorem{proposition}{Proposition}
\newtheorem{lemma}{Lemma}
\theoremstyle{definition}

\newtheorem{example}{Example}

\newtheorem{remark}{Remark}

\newcommand{\enm}[1]{\ensuremath{#1}}          %

\newcommand{\cal}[1]{\mathcal{#1}}

\newcommand{\Nm}{\mathrm{Num}}

\newcommand{\CC}{\enm{\mathbb{C}}}
\newcommand{\II}{\enm{\mathbb{I}}}

\newcommand{\RR}{\enm{\mathbb{R}}}
\newcommand{\QQ}{\enm{\mathbb{Q}}}

\newcommand{\FF}{\enm{\mathbb{F}}}

\newcommand{\PP}{\enm{\mathbb{P}}}

\newcommand{\Cc}{\enm{\cal{C}}}
\newcommand{\Dd}{\enm{\cal{D}}}

\newcommand{\Rr}{\enm{\cal{R}}}

\renewcommand{\phi}{\varphi}
\renewcommand{\theta}{\vartheta}
\renewcommand{\epsilon}{\varepsilon}

\renewcommand{\to}[1][]{\xrightarrow{\ #1\ }}

\newcommand{\old}[1]{}

\date{}

\begin{document}

\title[numerical range]
{On the numerical range of square matrices with coefficients in a degree $2$ Galois field extension}
\author{E. Ballico}
\address{Dept. of Mathematics\\
 University of Trento\\
38123 Povo (TN), Italy}
\email{ballico@science.unitn.it}
\thanks{The author was partially supported by MIUR and GNSAGA of INdAM (Italy).}

\begin{abstract}
Let $L$ be a degree $2$ Galois extension of the field $K$ and $M$ an $n\times n$ matrix with coefficients in $L$. Let $\langle \ ,\ \rangle :
L^n\times L^n\to L$ be the sesquilinear form associated to the involution $L\to L$ fixing $K$. We use $\langle \ ,\ \rangle$ to define the
numerical range $\Nm (M)$ of $M$ (a subset of $L$), extending the classical case $K=\RR$, $L=\CC$ and the case of a finite field introduced by Coons, Jenkins, Knowles, Luke and Rault.
There are big differences with respect to both cases for number fields and for all fields in which the image of the norm map $L\to K$ is not closed by addition, e.g.,
$c\in L$ may be an eigenvalue of $M$, but $c\notin \Nm (M)$. We compute $\Nm (M)$ in some case, mostly with $n=2$.
\end{abstract}

\maketitle

For any integer $n>0$ and any field $L$ let $M_{n,n}(L)$ be the $L$-vector space of all $n\times n$ matrices with coefficients in $L$.
Let $K$ be a field and $L$ a degree $2$ Galois extension of $K$. Call $\sigma$ the generator of the Galois group of the extension $K\hookrightarrow L$. Thus $\sigma : L\to L$
is a field isomorphism, $\sigma ^2: L\to L$ is the identity map and $K = \{t\in L\mid \sigma (t)=t\}$. For any $u=(u_1,\dots ,u_n)\in L^n$, $v=(v_1,\dots ,v_n)\in L^n$
set $\langle u,v\rangle := \sum _{i=1}^{n} \sigma (u_i)v_i$. The map $\langle \ ,\ \rangle : L^n\times L^n\to L$ is sesquilinear, i.e. for all $u, v, w\in L^n$ and all $c\in L$
we have $\langle u+v,w\rangle = \langle u,w\rangle + \langle v,w\rangle$, $\langle u,v
+w\rangle = \langle u,v\rangle + \langle u,w\rangle$, $\langle cu,w\rangle = \sigma ({c})\langle u,w\rangle$ and $\langle u,cw\rangle = c\langle u,w\rangle$.
Set $C_n(1):= \{u\in L^n\mid \langle u,u\rangle =1\}$. For any $M\in M_{n,n}(M)$ set $\Nm (M):= \{\langle u,Mu\rangle \mid u\in C_n(1)\}$. Since $C_n(1) \ne \emptyset$, we have
$\Nm (M)\ne \emptyset$. As in the classical case when $K =\RR$,
$L=\CC$ and $\sigma$ is the complex conjugation the subset $\Nm (M)$ of $L$ is called the {\emph{numerical range}} of $M$ (\cite{gr}, \cite{hj}, \cite{hj1}, \cite{pt}). When $K$ is a finite field  the numerical range
was introduced in \cite{cjklr} and \cite{b}. In particular \cite{cjklr} built a bridge between the classical case and the finite field case and at certain points we will duly quote
the parts of \cite{cjklr}, which we adapt to our set-up.

Assume for the moment $L =K(i)$ with $K\subset \RR$ and $\sigma$ the complex conjugation. In this case, calling $\Nm (M)_\CC \subset \CC$ the usual numerical range of $M$,
we have $\Nm (M) \subseteq \Nm (M)_\CC$ and hence $\Nm (M)$ is a bounded subset of $\CC$.
But even in this case there are many differences, in particular as for number fields not every element of $K$ is a square. The main differences come from the structures of the sets $\Delta$ and $\Delta _n$ defined below.

Let $\Delta \subseteq K$ be the image of the norm map $\mathrm{Norm}_{L/K}: L\to K$, i.e. set $\Delta := \{a\sigma (a)\mid a\in L\}\subseteq K$. If $a\in K$, then $\sigma (a)=a$
and hence $\mathrm{Norm}_{L/K}(a) =a^2$. Thus $\Delta $ contains all squares of elements of $K$. In particular $0\in \Delta$ and $1\in \Delta$. Since the norm map $\mathrm{Norm}_{L/K}$ is multiplicative, $\Delta$ is closed under multiplication. If $c\in \hat{\Delta}:= \Delta \setminus \{0\}$, say $c=\sigma (a)a$ for some
$a\in L\setminus \{0\}$, then $1/c = \sigma (a^{-1})a^{-1}$ and hence $\hat{\Delta}$ is a multiplicative group. For any integer $n>0$ let $\Delta _n$ be the set of all sums of $n$ elements of $\Delta$. If $K=\RR$,
then $\Delta = \Delta _n =\RR _{\ge 0}$ for all $n\ge 1$. If $K=\FF _q$ is a finite field, then $\Delta = \FF _q$, because in this case the norm map is surjective (\cite[Remark 3]{b}); hence $K=\Delta =\Delta _n$ if $K$ is a finite field. If $K=\QQ$ and $L =\QQ (i)$, then $\Delta \subsetneq \Delta _2$ (Example \ref{w4.00}).

For any $\delta \in \Delta _n$ set $C_n(\delta ):= \{u\in L^n\mid \langle u,u\rangle =\delta\}$. We have $L^n = \sqcup _{\delta \in \Delta _n} C_n(\delta )$ and $C_n(\delta )\ne \emptyset$
for all $\delta \in \Delta _n$. 

For any $M = (m_{ij})\in M_{n,n}(L)$ let $M^\dagger$ be the matrix $M^\dagger = (n_{ij})$ with $n_{ij}=\sigma (m_{ij})$ for all $i, j$. We have $(M^\dagger )^\dagger =M$
and $\langle u,Mv\rangle = \langle M^\dagger u,v\rangle$ for all $u, v\in L^n$. We say that $M$ is unitary if $M^\dagger M =\II _{n,n}$ (where $\II _{n,n}$ is the identity $n\times n$-matrix),
i.e. if $M^\dagger = M^{-1}$. For any $U, M\in M_{n,n}(L)$ with $U$ unitary, we have $\Nm (U^\dagger MU) =\Nm (M)$. In the case $n=1$, say $M = (m_{11})$, we have
$\Nm (M) =\{m_{11}\}$. We have $\Nm (c\II _{n,n}) =\{c\}$ for every $c\in L$.
For any $\mu \in L$ and $c\in \Delta$, the \emph{circle with center $\mu$ and squared-radius $c$} is the set of all $z\in L$ such that $\sigma (z-\mu)(z-\mu) =c$.
This set is never empty, since it contains the points $\mu +b$, where $b\in L$ is such that $\sigma (b)b =c$ (two points, $b$ and $-b$, if $\mathrm{char}(K) \ne 2$ and $b\ne 0$). If $c=0$, then the circle is just $\{\mu\}$, the center. If $c\in \hat{\Delta}$, then $b\ne 0$ and hence (assuming $\mathrm{char}(K)\ne 2$) this circle has at least two points, $\mu +b$ and
$\mu -b$. Hence if $c\ne 0$, this circle is a smooth conic and (if $K$ is infinite) it contains infinitely many points (Lemma \ref{c1} and, if $\mathrm{char}(K)=2$, Example \ref{2c-1}). See section \ref{Sc} for more and in particular for its description if $K=\QQ$ and so $L$ is a quadratic number
field.

For any integer $n>0$ let $\hat{\Delta}_n$ denote the sum of $n$ elements of $\hat{\Delta}$. Note that $0\in \hat{\Delta}_2$
if and only if there is $a\in \hat{\Delta}$ with $-a\in \hat{\Delta}$. In the case $n=1$ each matrix is a diagonal matrix and each numerical range is a singleton. The case $n>1$ is more complicated and interesting. We prove the following results.

\begin{proposition}\label{i1}
Assume $\mathrm{char}(K)=0$. If $M\in M_{n,n}(L)$ and $\Nm (M)= \{c\}$ for some $c\in L$, then $M =c\II _{n\times n}$.
\end{proposition}

In the classical case any eigenvalue of $M\in M_{n,n}(\CC)$ is in its numerical range. When either $0\in \hat{\Delta}_2$ or $\Delta _2\ne \Delta $, then this is not always the case, as shown by the following theorems \ref{i2} and \ref{a3}.

\begin{theorem}\label{i2}
Assume $0\in \hat{\Delta}_2$ and take $c\in L$ and $\mu \in L^\ast$. Then there is $M\in M_{2,2}(L)$ with $c$ an eigenvalue of $M$, $c\notin \Nm (M)$
and $\Nm (M) =c+\mu \hat{\Delta}$.
\end{theorem}

See Proposition \ref{i2.1} for a description of the matrices $M$ giving Theorem \ref{i2}. We have $0\in \hat{\Delta}_2$ for some real quadratic number fields (Lemma \ref{cz1}).

If $M$ has an eigenvalue $a$ with eigenvector $u$ with $\langle u,u\rangle \in \hat{\Delta}$, then $a\in \Nm (M)$ (Remark \ref{i0.1.0}).

Part (a) of the following result is an adaptation of \cite[Theorem 1.2 ({c})]{cjklr}
\begin{theorem}\label{a3}
Assume $n=2$ and that $M$ has a unique eigenvalue, $c$. Assume that $c$ has an eigenvector $v$ with $\delta := \langle v,v\rangle \ne 0$ and that $M \ne c\II _{2,2}$. 

\quad (a) $c\in \Nm (M)$ if and only if $\delta \in \Delta$.

\quad (b) Assume $\delta \in \Delta$. There is $\mu \in L^\ast$ such that $(\Nm (M) -c)/\mu$ is the union of $\{0\}$ and all all circles $C(k(1-k), 0)$ with $k\in \hat{\Delta}\cap (1-\hat{\Delta})$. \end{theorem}
For any $\delta \in \Delta _2\setminus \{0\}$ and $v\in L^2$ with $\langle v,v\rangle =\delta$ the set of all $M$ as in Theorem \ref{a3} are exactly the matrices $M$ such that $\mathrm{Ker}(M-c\II _{2\times 2}) =\mathrm{Im}(M) =Lv$.

In section \ref{Scc} we consider the case $M\in M_{n,n}(K)$. Set $C_n(1,K):= \{(x_1,\dots ,x_n)\in K^n\mid x_1^2+\cdots +x_n^2=1\}$. Note that $C_n(1,K):= C_n(1)\cap K^n$. Note
that $C_n(1,K)\ne 0$ (e.g. take $x_i=1$ and $x_j=0$ for all $j\ne i$).
The {\emph{$K$-numerical range}} $\Nm (M)_K$ of $M$ is the set of all $\langle u,Mu\rangle$ with $u\in C_n(1,K)$. We have $\Nm (M)_K\subseteq K$. The case $\mathrm{char}(K) =2$
is quite different (and easier) from the case $\mathrm{char}(K) \ne 2$.

\begin{proposition}\label{cc1}
Assume $\mathrm{char}(K) \ne 2$. Take $M\in M_{n,n}(K)$, $n>1$. We have $\Nm (M)_K = \{c\}$ if and only if the matrix $M-c\II _{n\times n}$ is antisymmetric.
\end{proposition}

\begin{proposition}\label{cc6}
Assume $\mbox{char}(K) =2$ and take $M=(m_{ij})\in M_{n,n}(K)$.

\quad (a) We have $\Nm (M)_K = \{c\}$ for some $c\in K$ if and only if $m_{ii}=c$ for all $i$ and $m_{ij}=m_{ji}$ for all $i\ne j$.

\quad (b) If $\sharp (\Nm (M)_K) \ne 1$ and $K$ is infinite, then $\Nm (M)_K$ and $K$ have the same cardinality.
\end{proposition}

\section{Circles}\label{Sc}

Let $\overline{L}$ be an algebraic closure of $L$. In this section we assume that $K$ is infinite and that $\mathrm{char}(K) \ne 2$ (see Example \ref{2c-1} for the case $\mathrm{char}(K)= 2$). With these assumptions there is $\alpha \in K$, which is not a square
and with $L =K(\sqrt{\alpha})$. Fix $\beta \in L$ such that $\beta ^2=\alpha$. In $\overline{L}$ the equation $t^2=\alpha$ has $\beta$ and $-\beta$ as its only solutions.
$L$ is a $2$-dimensional $K$ vector space over $K$ with $1$ and $\beta$ as its basis. Hence for any $z\in L$ there are uniquely determined $x, y\in K$
such that $z =x+y\beta$. Since $\sigma (\beta )=-\beta$ and $\sigma (t)=t$ for every $t\in K$, we have $\sigma (z) = x-y\beta$ and hence
$\sigma (z)z = x^2-y^2\alpha$. Take $k, \mu \in L$. The map $z\mapsto z-\mu$ induces a bijection between  the set $\{z\in L\mid \sigma (z-\mu )(z-\mu ) =k\}$
and the set $G(k,0):= \{z\in L\mid \sigma (z)z=k\}$. Hence it is sufficient to study the circles with center $0\in L$. By the definition of $\Delta $ if $k\notin \Delta $, then $G(k,0) =\emptyset$, while if $k\in \Delta $ we have $G(k,0) \ne \emptyset$. We have $G(0,0) =\{0\}$, because $\sigma (z)z =0$ if and only if $z=0$.
Write $z = x+y\beta$ and hence $\sigma (z) = x-y\beta$ and $\sigma (z)z =x^2-\alpha y^2$. Thus $G(k,0) =\{(x,y)\in K^2 \mid x^2-\alpha y^2=k\}$.
Now assume $k\in \hat{\Delta} =\Delta \setminus \{0\}$. Write $k= \sigma ({c})c$ for some $c\in L^\ast$. Note that $\sigma (z)z=c$ if and only either $z=c$ or $z=-c$.
Since $\mathrm{char}(K)\ne 2$, the set $G(k,0)$ contains at least two points, $-c$ and $c$.

\begin{lemma}\label{c1}
If $k\in \hat{\Delta}$ the circle $G(k,0) \subset L=K^2$ is a smooth affine conic over $K$. If $K$ is infinite, then $G(k,0)$ and $K$ have the same cardinality.
\end{lemma}

\begin{proof}
Write $k= \sigma ({c})c$ for some $c\in L^\ast$. We saw that $G(k,0)$ contains the points $c$ and $-c$ and in particular $G(k,0)\ne \emptyset$. See $x, y, z$ as homogeneous variables of $\PP^2(K)$, with the line $\ell _\infty = \{z=0\}$ as the set $\PP^2(K)\setminus K^2$. Let $D(k,0)\subset \PP^2(K)$ be the conic with
$g(x,y,z):= x^2-\alpha y^2-kz^2$ as its equation. The linear forms $2x$, $-2\alpha y$ and $-2kz$ are the partial derivatives of $g(x,y,z)$. Set $D(k,0)_{\overline{K}}:= \{(x:y:z)\in \PP^2(\overline{L})\mid g(x,y,z)=0\}$.
Since $\alpha \ne 0$, $k\ne 0$ and $\mathrm{char}(K)\ne 2$, the partial derivatives of $g(x,y,z)$ have no common zero in $\PP^2(\overline{L})$. Thus $g(x,y,z)$
is irreducible and $D(k,0)_{\overline{L}}$ is a smooth conic. Hence $D(k,0)$ is a smooth conic defined over $K$. Since $D(k,0)$ has a $K$-point, $c$, $D(k,0)$ is isomorphic
to $\PP^1_K$ (use the linear projection from $c$) and in particular (for infinite $K$), $K$ and $D(k,0)$ have the same cardinality. The set $D(k,0)\cap \ell _\infty$ has at most two points, because $g(x,y,z)$ is irreducible and so
$ \ell _\infty$ is not a component of $D(k,0)$. Thus (since $K$ is infinite) $G(k,0)$ and $K$ have the same cardinality.
\end{proof}

\begin{remark}\label{c2}
Take $\ell _\infty$, $D(k,0)$ and $g(x,y,z):= x^2-\alpha y^2-kz^2$ as in the proof of Lemma \ref{c1}. We saw that $G(k,0) = D(k,0)\setminus \ell _\infty \cap D(k,0)$. Here we check
that $\ell _\infty \cap D(k,0) =\emptyset$, i.e. $G(k,0)=D(k,0)$. We have $\ell _\infty \cap D(k,0) =\{(x:y:0)\in \PP^2(K)\mid x^2-\alpha y^2=0\}$.
Since $\alpha$ is not a square in $K$, if $(x,y)\in K^2$ and $x^2=\alpha y^2$, then $x=y=0$.
\end{remark}

\begin{example}\label{c3}
Take $K=\QQ$. Hence $L$ is a quadratic number field. There is a unique square-free integer $d\notin \{0,1\}$ such that $L =\QQ(\sqrt{d})$ (\cite[Ch 13, \S 1]{ir}). Take $k\in \hat{\Delta}$.
If $d>0$, then $G(k,0)$ is a hyperbola with infinitely many points and it is unbounded. If $d<0$, then $G(k,0)$ is an ellipsis and in particular it is bounded; hence each $\Delta _n$ is bounded.
\end{example}

\section{Lemmas and examples}\label{Sp}
For any field $F$ set $F^\ast:= F\setminus \{0\}$. Let $e_1 =(1,0,\dots ,0),\dots ,e_n =(0,\dots ,0,1)$ be the standard basis of $L^n$. For any $M\in M_{n,n}(L)$ let $\Nm _0(M) \subseteq L$ be the union of all $\langle u,Mu\rangle$ with $\langle u,u\rangle =0$. 

\begin{remark}\label{i3}
Take $M = (m_{ij})\in M_{n,n}(L)$. Since $m_{ii} =\langle e_i,Me_i\rangle$, all diagonal elements of $M$
are contained in $\Nm (M)$.
\end{remark}

\begin{remark}\label{w0}
Fix $\delta \in \Delta _n$ and $a\in \Delta \setminus \{0\}$. Take $b\in L$ such that $a =b\sigma (b)$. For any $u\in L^n$ we have
$\langle bu,bu\rangle =a\langle u,u\rangle$ and hence $C_n(a\delta ) =bC_n(\delta)$.
\end{remark}

\begin{remark}\label{w1}
Since $\sigma (x)=x$ for all $x\in K$, $\Delta$ contains all squares in $K$.
\end{remark}

\begin{remark}\label{a4}
For any $M\in M_{n,n}(L)$ and any $c, d\in L$ we have $\Nm (cM+d\II _{n\times n}) = d +c\Nm (M)$.
\end{remark}

\begin{lemma}\label{w2}
 Fix $c\in \Delta _n \setminus \{0\}$. Then $1/c\in \Delta _n$.
 \end{lemma}
 
 \begin{proof}
 If $c =\sigma (a_1)a_1+\cdots +\sigma (a_n)a_n$ with $a_i\in K$, then $1/c = \sigma (a_1/c)a_1/c+\cdots +\sigma (a_n/c)a_n/c$.
 \end{proof}
 
 \begin{lemma}\label{a5}
For any $M\in M_{n,n}(L)$ we have $\Nm (M^\dagger )=\sigma (M)$.
\end{lemma}

\begin{proof}
For any $u\in C_n(1)$ we have $\langle u,Mu\rangle = \langle M^\dagger u,u\rangle = \sigma (\langle u,M^\dagger u\rangle )$.
\end{proof}

 \begin{lemma}\label{i0.1}
Fix $u\in L^n$ and assume $\delta:= \langle u,u\rangle \ne 0$. There is $t\in L^\ast$ such that $\langle tu,tu\rangle =1$ if and only if $\delta \in \Delta$.
\end{lemma}

\begin{proof}
First assume the existence of $t\in  L^\ast$ such that $\langle tu,tu\rangle =1$. We have $\langle tu,tu\rangle =\sigma (t)t\delta$. Since $t\ne 0$, $\sigma (t)t\in \hat{\Delta}$. Remarks 
\ref{w1} and \ref{w2} give $\delta \in \Delta$. Now assume $\delta \in \Delta$. Since $\delta \ne 0$, we have $1/\delta \in \hat{\Delta}$ (Remark \ref{w2}). Write $1/\delta =\sigma (t)t$
for some $t\in L^\ast$. We have $\langle tu,tu\rangle =1$.
\end{proof}

\begin{remark}\label{i0.1.0}
Take $M\in M_{n,n}(L)$ with an eigenvector $v$ (say $Mv=cv$) such that $\langle v,v\rangle \in \hat{\Delta}$. Lemma \ref{i0.1} gives $c\in \Nm (M)$:
\end{remark}

\begin{lemma}\label{cz1}
Assume $\mathrm{char}(K)\ne 2$ and take $L =K(\sqrt{\alpha})$ with $\alpha$ not a square in $K$, but $\alpha$ the sum of two squares in $K$.
Then $0\in \hat{\Delta}_2$.
\end{lemma}

\begin{proof}
Note that $0\in \hat{\Delta}_2$ if and only if there is $a\in \hat{\Delta}$ with $-a\in \hat{\Delta}$. Write $\alpha = u^2+v^2$ with $u,v\in K^\ast$. Take $a :=u^2 = -(v^2-\alpha )$.
\end{proof}

 \begin{lemma}\label{w0.1}
 Fix integers $n> m> n/2 >1$ and assume $\Delta _n = \Delta$. Let $M\subset L^n$ be an $m$-dimensional $L$-linear subspace. Then there
 are $f_1,\dots ,f_{3m-2n}\in M$ such that $\langle f_i,f_i\rangle =1$ for all $i$ and $\langle f_i,f_j\rangle =0$ for all $i\ne j$.
 \end{lemma}
 
 \begin{proof}
 Take any basis $u_1,\dots ,u_m$ of $M$ and complete it to a basis $u_1,\dots ,u_n$ of $L^n$. Since the sesquilinear form $\langle \ ,\ \rangle$ is non-degenerate,
 the matrix $E = (a_{ij})$ with $a_{ij} = \langle u_i,u_j\rangle$ has rank $n$. Hence among the first $m$ rows of $E$, at least $2m-n$ are linearly independent.
 Hence the $m\times m$ matrix $(a_{ij})$, $i, j=1,\dots ,m$, has rank at least $3m-2n$. Permuting $u_1,\dots ,u_m$ we may assume that the matrix $(a_{ij})$, $i,j=1,\dots ,3m-2n$,
 is non-singular. Let $W\subset M$ be the linear span of $u_1,\dots ,u_m$. Since the matrix $(a_{ij})$, $i,j=1,\dots ,3m-2n$,
 is non-singular, the restriction $\langle \ ,\ \rangle _W$ of $\langle \ ,\ \rangle$ to $W$ is non-degenerate. Hence there is $g_1\in W$ with $\langle g_1,g_1\rangle \ne 0$.
 Since $\Delta _n=\Delta$, there is $t\in L$ such that $\langle tg_1,tg_1\rangle = 1$ (Lemma \ref{i0.1}). Set $f_1:= tg_1$. If $3m-2n >1$ set $W_1:= \{w\in W\mid \langle f_1,w\rangle =0\}$. $W_1$ is a codimension $1$ linear subspace of $W$ and the restriction of $\langle \ ,\ \rangle $ to $W_1$ is non-degenerate. Therefore there is $g_2\in W_1$
 with $\langle g_2,g_2\rangle \ne 0$. Take $z\in L$ such that $\langle zg_2,zg_2\rangle =1$ (Lemma \ref{i0.1}) and set $f_2:= zg_2$. If $3m-2n>2$ set $W_2:= \{w\in W_1\mid \langle f_2,w\rangle =0\}$ and continue in the same way. \end{proof}

The definition of numerical range and of unitary direct sum immediately gives the following lemma.
\begin{lemma}\label{a1}
Fix integers $n>x>0$, $A\in M_{x,x}(L)$ and $B\in M_{n-x,n-x}(L)$. Set $M:= A\oplus B \in M_{n,n}(L)$ (unitary direct sum).
The set $\Nm (M)$ is the union of all points $a+b$ of $L$ with $a = \langle u,Bu\rangle$, $b= \langle v,Av\rangle$,
$u\in L^x$, $v\in L^{n-x}$ and $\langle u,u\rangle +\langle v,v\rangle =1$.
\end{lemma}

When $\Delta = \Delta _n$ we may improve Lemma \ref{a1} in the following way. 

\begin{proposition}\label{a2}
Fix integers $n>x>0$, $A\in M_{x,x}(L)$ and $B\in M_{n-x,n-x}(L)$. Set $M:= A\oplus B \in M_{n,n}(L)$ (unitary direct sum). Assume
$\Delta =\Delta _x =\Delta _{n-x}$.
Then $\Nm (M)$ is the union of  $\{\Nm _0(A) +\Nm (B)\}\cup \{\Nm (A) +\Nm _0(B)\}$ and all $tc+(1-t)d$ with $t\in \hat{\Delta}\cap (1-\hat{\Delta})$,
$c\in \Nm (A)$ and $d\in \Nm (B)$.
\end{proposition}

\begin{proof}
Fix $u\in L^x$, $v\in L^{n-x}$ with $\langle u,u\rangle +\langle v,v\rangle =1$. Set $t = \langle u,u\rangle$. Hence $\langle v,v\rangle =1-t$. If
$t=0$ (resp. $t=1$), then $1-t =1$ (resp. $1-t=0$) and hence $\{\Nm (A)_0 +\Nm (B)\}\cup \{\Nm (A) +\Nm (B)_0\} \subseteq \Nm (M)$. Now assume
$t\in \Delta \setminus \{0,1\}$. Since $\Delta = \Delta _x =\Delta _{n-x}$, we have $t\in \hat{\Delta}\cap (1-\hat{\Delta})$.
Since $\{t, 1- t\}\subset \hat{\Delta}$,  there are $c,d\in L^\ast$ with $\sigma ({c})c =1/t$ and $\sigma (d)d=1/(1-t)$.
Set $\alpha := \langle u,Au\rangle$ and $\beta := \langle u,Bu\rangle$. 
We have $\langle cu,cu\rangle = \langle dv,dv\rangle =1$ and hence  $\langle cu,cAu\rangle \in \Nm (A)$ and $\langle dv,dBv\rangle \in Nm (B)$.
Thus $\alpha /t \in \Nm (A)$ and $\beta /(1-t)\in \Nm (B)$. We get $\langle u+v,M(u+v)\rangle = tx+(1-t)y$ with $x\in \Nm (A)$ and $y\in \Nm (B)$.
The same proof done backwards gives the other inclusion.
\end{proof}

Proposition \ref{a2} is analogous to \cite[Proposition 3.1]{cjklr}. Fix $c,d\in L$. The set  of all $tc+(1-t)d$ with $t\in (\hat{\Delta}\cap (1-\hat{\Delta}))$ is called in \cite{cjklr} the \emph{open segment} with $c$ and $d$ as its boundary points
and we denote it with $((c;d))$.
When (as in the case of $\mathrm{char}(L)=0$) the set $(\hat{\Delta}\cap (1-\hat{\Delta}))$ is non-empty (Lemma \ref{a4.1}) we have $((c;c)) =\{c\}$ for all $c\in L$.

\begin{lemma}\label{a4.1}
Assume $\mathrm{char}(K)=0$. Then $\hat{\Delta}\cap (1-\hat{\Delta})$ is infinite.
\end{lemma}

\begin{proof}
For any $\delta \in \Delta$ there are $x, y$ in $K$ such that $x^2-\alpha y^2 = \delta$. Note that $1-\delta \in \Delta$ if and only if there are $w,z\in K$ such that $1-\delta = w^2-\alpha z^2$. Take coordinates $(x,y,w,z)$ on $K^4$. Set $T:= \{(x,y,w,z)\in K^4\mid x^2+w^2-\alpha (y^2+z^2)=1\}$. We take homogeneous coordinates $x,y,w,z,t$ in $\PP^4(K)$ with $H_\infty
=\{t=0\}$ and $K^4=\PP^4(K)\setminus H_\infty$.
Let $E\subset \PP^4(K)$ the projective quadric with equation $\{x^2+w^2-\alpha (y^2+z^2) -t^2=0\}$. We have $E\setminus E\cap H_\infty =T$. Since $\mathrm{char}(K)\ne 2$
and $\alpha \ne 0$, taking the partial derivatives of the polynomial $x^2+w^2-\alpha (y^2+z^2) -t^2$ we get that
the point $O:= (1:0:0:0:1)$ is a smooth point of $T$. Let $M\subset \PP^4$ be the hyperplane with equation $x-t =0$. Note that $M$ is the tangent space to $E$ at $O$.
Hence $E\cap M$ is a quadric cone of $M$, which is the union of all lines of $\PP^4$ contained in $E$ and passing through $O$. Let $H\subset \PP^4(K)$ be any hyperplane defined over $K$ and with $O\notin H$. The latter condition implies $H\ne M$ and hence $N:= H\cap M$ is a codimension two linear subspace of $\PP^4$. Let $\ell : \PP^4\setminus \{O\}\to H$
denote the linear projection from $O$. The morphism $\ell$ is defined over $K$, because $O$ and $H$ are defined over $K$. Hence for each $P\in H(K)$
the line $L(O,P)$ spanned by $O$ and $P$ is defined over $K$. Since $O\in T$, the intersection $T\cap L(O,P)$ is either $O$ with multiplicity $2$ or the entire line $L(O,P)$
or the union of $O$ and another point $O_P\in E$ defined over $K$. The first two cases imply $L(O,P)\subset M$. Since $O\in M$ and $O\notin H$, we have $L(O,P)\subset M$
if and only if  $P\in N$. Hence whenever we take $P\in H\setminus N$ the point $O_P\in E\setminus \{O\}$ is defined over $K$. Since $H\setminus N$ is $3$-dimensional
affine space over $K$, we get that $E$ is infinite. $E\setminus T = E\cap H_\infty$. We have
$O\notin H_\infty$ and hence $O\notin H_\infty \cap E$. Thus $\ell (H_\infty \cap E)$ is a quadric hypersurface of $M$.
If $P\in M\setminus (\ell (H_\infty \cap E))$, then $O_P\in T$. $\ell (H_\infty \cap E)\cup N$ is the union of a quadric and a hyperplane of $M$. Since $K$ is infinite, the Grassmannian of all lines of $M(\overline{L})$
defined over $K$ is Zariski dense in the Grassmannian of all $\overline{L}$-lines of $M(\overline{L})$. Since $K$ is infinite, restricting to lines
defined over $K$ and contained neither in $\ell (H_\infty \cap E)$ nor in $N$ we get that $M\setminus (N\cup \ell (E\cap H_\infty))$ is infinite. Hence $E$ is infinite.

\quad (a) Assume that $L$ has a field embedding $j: L\hookrightarrow \CC$. We omit $j$ and hence see $L$ as a subfield of $\CC$. 

First assume that $K$ is dense in $\CC$ with respect to the euclidean topology. Hence $K^4$ (resp. $\PP^4(K)$) is dense in $\CC ^n$ (resp. $\PP^4(\CC)$) for the euclidean topology.
The topological space $N(\CC )$ is the closure of $N$ in $N(\CC)$ with respect to the euclidean topology. Since $E\cap H_\infty$ has corank $1$ with vertex $O\in \PP^4(K)$, the closure
of $E\cap H_\infty$ in the euclidean topology contains a neighborhood of $O$ in $(E\cap H_\infty )(\CC)$. Since $(E\cap H_\infty )(\CC)$ is a cone
with vertex $O$, it is the closure of $E\cap H_\infty$ for the euclidean topology and $(\ell (H_\infty \cap E))(\CC)$ is the closure of $\ell (H_\infty \cap E)$
for the euclidean topology. $\ell (H_\infty \cap E)\cup N$ is the union of a quadric and a hyperplane of $M$. We get that $E(\CC)$ is the closure of $E$ with respect to the euclidean topology. $\hat{\Delta}\cap (1-\hat{\Delta})$ is infinite if and only if $\Delta \cap (1-\Delta )$ is infinite. Assume that  $\Delta \cap (1-\Delta )$ is finite, say  $\Delta \cap (1-\Delta ) =\{a_1,\dots ,a_s\}$ with $a_i\in K$. Set $G_i:= C(a_i,0)$ and $F_i:= C(1-a_i,0)$. We get $E = \cup _{i=1}^{s} G_i\times F_i$. Hence $ \cup _{i=1}^{s} G_i(\CC )\times F_i(\CC)$ is
dense in $E(\CC)$ for the euclidean topology. Since $E(\CC )$ has complex dimension $3$, while each $G_i(\CC )\times F_i(\CC)$ has complex dimension $2$, we get
a contradiction. 

Now assume that $K$ is not dense in $\CC$ for the euclidean topology. Since $\QQ$ is dense in $\RR$ for the euclidean topology, the closure
of $K$ for the euclidean topology contains $\RR$. Since this closure is a field, $\RR$ is the closure of $K$ for the euclidean topology. We use $E(\RR )$ instead of $E(\CC )$.
Since $O$ is a smooth point of $E$ and $O\in E(\RR )$, $E(\RR )$ is a non-empty topological manifold of dimension $3$. Hence $E(\RR )$ cannot be the union of finitely
many topological $2$-manifolds $G_i(\RR )\times F_i(\RR )$.

\quad (b) By a theorem of Steinitz two algebraically closed fields with characteristic zero are isomorphic if and only if they have transcendental basis over $\QQ$
with the same cardinality (\cite[Theorem VIII.1.1]{l}, \cite[page 125]{s}). There are real closed field with a transcendental basis over $\QQ$ with arbitrary cardinality (use
that every ordered field has a real closure (\cite[Theorem 1.3.2]{bcr}) and give an ordering of $\QQ ({t_\alpha}) _{\alpha \in \Gamma}$ with $\Gamma$ a well-order set and $t_\alpha$ bigger than any rational function in the variable $t_\gamma$, $\gamma <\alpha$) and for any real closed field $\Rr$ the
field $\Rr (i)$ is algebraically closed (\cite[Theorem 1.2.2]{bcr}). Hence there is an embedding $j: L\hookrightarrow \Rr (i)$ for some real closed field $\Rr$. The euclidean topology on $\Rr^n$ is the topology for which open balls form a basis of open subsets (\cite[Definition 2.19]{bcr}). The field $\Cc:= \Rr (i)$ inherits the euclidean topology.  The sets $\Rr^n$, $\Cc ^n$, $T(\Rr)$, $T(\Cc )$, $\PP^r(\Rr)$, $\PP^r(\Cc )$, $E(\Rr)$ and $E(\Cc)$, have the euclidean topology. Repeat the proof in step (a) with $\Rr$ and $\Cc$ instead of $\RR$
and $\CC$.
\end{proof}

\begin{remark}\label{a4.1.1}
Assume $\mathrm{char}(K) =0$. Lemma \ref{a4.1} says that $C_2(1)$ is infinite. Hence $C_n(1)$ is infinite for all $n\ge 2$.
\end{remark}

We recall that a field $F$ is said to be \emph{formally real} if $-1$ is not a sum of squares of elements of $F$. If $F$ is formally real, then $\mathrm{char}(F)=0$.

\begin{proposition}\label{zc2}
Assume that $K$ is formally real, but that $L$ is not formally real. Then $0\notin \hat{\Delta}_n$ for any $n>1$.
\end{proposition}

\begin{proof}
Write $L = K(\sqrt{\alpha})$ for some $\alpha \in K$. Since $K$ is formally real, but $L$ is not formally real, there is an ordering $\le$ on $K$ with $\alpha <0$
(\cite[Theorem 1.1.8 and Lemma 1.1.7]{bcr}).
Take $z = x+y\alpha \in L$ with $x, y\in K$ and $(x,y)\ne (0,0)$. Since $\sigma (z)z = x^2-\alpha y^2 >0$, $a>0$ for every $a\in \hat{\Delta}$. Thus $b>0$ for
every $b\in \hat{\Delta}_n$.
\end{proof}

\begin{example}\label{w4.00}
Here we give a simple example with $\Delta _2\ne \Delta$ and $0\notin \hat{\Delta}$. Take $K=\QQ$ and $L:= \QQ (i)$. For any $z=x+iy\in L$ we have $z\sigma (z) = x^2+y^2$. Hence
$\Delta$ is the subset of $\QQ _{\ge 0}$ formed by the sums of two squares. Every positive integer is the sum of $4$ squares by a theorem of Lagrange and
hence $\Delta _2 =\QQ _{\ge 0}$. There are positive rational numbers, which are not the sums of $2$ or $3$ squares, e.g. $7$ is not the sum of $3$
squares of rational numbers, because for any $n\in \{1,2,3\}$ a positive integer is the sum of $n$ squares of rational numbers if and only if it is the sum of $n$ squares of 
integers (\cite[Ch. 7]{Lev}). Proposition \ref{zc2} gives $0\notin \hat{\Delta}$.
\end{example}

\begin{example}\label{ccc1}
Assume $\mathrm{char}(K) =0$, i.e. assume $K\supseteq \QQ$. We have $\Delta _n \supseteq \QQ _{\ge 0}$ for every $n\ge 4$, because every non-negative integer
is the sum of $4$ integers by a theorem of Lagrange and $\Delta _n\setminus \{0\}$ is a multiplicative group.
\end{example}

\begin{lemma}\label{ccc2}
Assume $\mathrm{char}(K) =0$. Then there are infinitely many $m\in \Delta$ such that $1-m$ is a square in $K$ and $m= \sigma (z)z$ with $z\in L\setminus (K\cup K\sqrt{\alpha})$.
\end{lemma}

\begin{proof}
Take $z = x+y\sqrt{\alpha}\in L^2$ with $x,y\in K$. Consider the equation in $K^3$:
\begin{equation}\label{eqccc1}
x^2+w^2 = \alpha y^2
\end{equation}
As the proof of Lemma \ref{a4.1} we get that (\ref{eqccc1}) has infinitely many solutions $(x,y,w)\in K^3$ with $y\ne 0$ and $x\ne 0$ and that the set of these
solutions has infinite image under the projection $K^3\to K$ onto the third coordinate.
\end{proof}

\begin{example}\label{2c-1}
Assume $\mathrm{char}(K) = 2$. We also assume that $K$ is infinite. There is $\epsilon \in K$ such that $L= K(\beta )$, where $\beta$ is any root of the polynomial
$t^2+t+\epsilon$. Note that $1+\beta$ is the other root of the same polynomial and hence $\sigma (\beta )=1+\beta$ (note that $\sigma ^2(\beta )=2
+\beta =\beta$). We use $1, \beta$ as a basis of $L$ as a $2$-dimensional $K$-vector space. Take $z =x+y\beta \in L$. We have $\sigma (z) =x+y +y\beta$ and hence (since $xy+yx=0$ and $\beta^2+\beta =\epsilon$) $\sigma(z)z= x(x+y) + y^2\epsilon$.
For any $c\in \Delta$ the affine conic $\{x^2+ xy +y^2\epsilon +c=0\}\subset K^2$ is the circle with center $0$ and squared-radius $c$. If $c=0$, then this circle is the singleton
$\{0\}$. Now assume $c\in \hat{\Delta}$.
Take homogeneous coordinates $x, y, z$ in $\PP^2(\overline{L})$. Fix any $c\in K^\ast$. Set $g(x,y,z):= x^2+xy +y^2\epsilon +cz^2$. The projective conic $T:= \{g(x,y,z) =0\}\subset
\PP^2(\overline{L})$ is smooth, because $\frac{\partial}{\partial _x}g =y$, $\frac{\partial}{\partial _y}g =x$, $\frac{\partial}{\partial _z}g =0$
and hence the common zero-set of the partial derivatives of $g(x,y,z)$ is the point $(0:0:1) \notin T$. Hence if $c\in \hat{\Delta}$ any circle
$\{\sigma (z-\mu)(z-\mu )=c\}$ is an affine smooth conic with at least one point $P$ (because $c\in \Delta$). Taking the linear projection from $P$ we see that this circle
is infinite and with the cardinality of $K$.
\end{example}

\section{The proofs and related results}\label{St}

\begin{proof}[Proof of Proposition \ref{i1}:]
Taking $M-c\II _{n\times n}$ instead of $M$ and applying Remark \ref{a4} we reduce to the case $c=0$. Assume $\Nm (M)= \{0\}$ and $M\ne 0\II _{n\times n}$. In particular we
have $n>1$. Write $M= (m_{ij})$. Since every diagonal element
of $M$ is contained in $\Nm (M)$ by Remark \ref{a4}, we have $m_{ii}=0$ for all $i$. Hence there is $m_{ij}\ne 0$ with $i\ne j$. Taking $M^\dagger $ instead of $M$ if necessary
and applying Lemma \ref{a5} we reduce to the case $i<j$. A permutation of the orthonormal basis $e_1,\dots ,e_n$ is unitary and hence it preserves $\Nm (M)$.
Permuting this basis we reduce to the case $i=1$ and $j=2$. Taking $(1/m_{12})M$ instead of $M$ and applying Remark \ref{a4} we reduce to the case $m_{12}=1$. Set
$b:= m_{21}$. 
Take $u =(x,y)\in L^2$ such that $\sigma (x)x+\sigma (y)y=1$ and set $u:= (a_1,\dots ,a_n)$ with $a_1=x$, $a_2=y$ and $a_i=0$ for all $i>2$. We have
$\langle u,Mu\rangle = b\sigma (x)y +\sigma (y)x$. Since $b\ne 0$, it is sufficient to prove the existence of $x,y\in L^\ast$ such that $f(x,y):= \sigma (x)x+\sigma (y)y= 1$
and $g(x,y):= b\sigma (x)y +\sigma (y)x \ne 0$. We first take $x,y\in K$ and so $\sigma (x)=x$ and $\sigma (y)=y$. Hence $g(x,y) =2(b+1)xy$. Since $\mathrm{char}(K) \ne 2$, the affine conic $D:= \{x^2+y^2=1\} \subset K^2$ is smooth. Since $K$ is infinite and $(1,0)\in D$, $D$ has infinitely many $K$-points (use the linear projection from $(1,0)$). Hence we may find $(x,y)\in D$ with $g(x,y)\ne 0$, unless
$b=-1$. Now assume $b=-1$. Write $y=tx$. We have $g(x,y) =0$ if and only if either $xy=0$ or $t\in K$. By Lemma \ref{ccc2} there is $y\in L\setminus K$, so that $1-\sigma (y)y$ is a square in $K^\ast$ and so there is $x\in K^\ast$, $y\in L\setminus K$ with $1=\sigma(y)y+x^2$.
\end{proof}

\begin{proposition}\label{i2.1}
Take $M\in M_{2,2}(L)$ with a unique eigenvalue $c\in L$ with an eigenvector $v\ne 0$ with $\langle v,v\rangle =0$. Then $0\in \hat{\Delta}_2$ and either $M = c\II _{2,2}$
or there is $\mu \in L^\ast$ such that $\Nm (M) = c+\mu \hat{\Delta}$ and in the latter case $c\notin \Nm (M)$.
\end{proposition}

\begin{proof}
Write $v:= (a_1,a_2)$. By assumption $(a_1,a_2)\ne (0,0)$ and
$\sigma (a_1)a_1+\sigma (a_2)a_2 =0$. Hence $0\in \hat{\Delta}_2$. Since $v\ne 0$ and $\langle v,v\rangle =0$, $e_2$ and $v$ are not proportional and so they form a basis of $L^2$. Since $\langle \ ,\ \rangle$ is non-degenerate and  $\langle v,v\rangle =0$, we have  $\langle v,e_2\rangle \ne 0$. Taking a multiple of $v$ if necessary
we reduce to the case $\langle v,e_2\rangle =1$. Thus $\langle e_2,v\rangle =1$. Assume $M\ne c\II _{2,2}$. Taking $M-c\II _{2,2}$ instead of $M$ and applying Remark \ref{a4}
we reduce to the case $c=0$. Write $M =(b_{ij})$, $i,j=1,2$, with respect to the basis $v,e_2$. We have $b_{11}=b_{22}=b_{12}=0$ and $b_{12}\ne 0$. Set $\mu:= b_{12}$.
Take $u=xv+ye_2$ such that $\langle u,u\rangle =1$, i.e. $x\sigma (y)+\sigma (x)y +\sigma (y)y=1$ (and in particular $y\ne0$). We have $\langle u,Mu\rangle = \langle xv+ye_2,\mu yv\rangle
= \mu \sigma (y)y$. Varying $y\in L^\ast$ as $\sigma (y)y$ we get all elements of $\hat{\Delta}$. Hence to conclude the proof it is sufficient to prove that for every $y\in L^\ast$
there is $x\in L$ such that $x\sigma (y) \sigma (x)y +\sigma (y)y=1$. First assume $\mathrm{char}(K) \ne 2$. Fix $e\in L\setminus K$ with $e^2 \notin K$ and write $x= u+ev$ and $y =c+ed$ with
$u, v, c,\in K$. We have $\sigma (x) = u-ev$ and $\sigma (y) =c-ed$. We find $2uc =\eta$ for some $\eta \in K$ and we always have a solution $u$, because we may take $y$ with
gives $\sigma (y)y$ and $c\ne 0$. Now assume $\mathrm{char}(K) =2$. There is $e\in L\setminus K$ such that $e^2+e \in L$ and $\sigma (e) =e+1$. Write
$x = u+ev$, $y=c+ed$ with $u, v,c, d\in K$ and $(c,d)\ne (0,0)$. Now we find an equation $ud+vc =\eta$ for some $\eta\in K$, with may always be satisfied, because $(c,d)\ne (0,0)$.
\end{proof}

\begin{remark}\label{i2.1bis}
Assume $0\in \hat{\Delta}_2$, say $0 = \sigma (x)x+\sigma (y)y$ with $(x,y)\in L^2\setminus \{(0,0)\}$, and take $c\in L$. Set $v:= xe_1+ye_2\in L^2\setminus \{(0,0)\}$.
Take any linear map $f: L^2\to L^2$ with $\mathrm{Ker}(f-c\mathrm{Id}_{L^2}) = \mathrm{Im}(f) =Lv$ and let $M$ be the matrix associated to $f$. Then $M \ne c\II _{2\times 2}$
and $M$ satisfies the assumptions of Proposition \ref{i2.1}.
\end{remark}

\begin{proof}[Proof of Theorem \ref{i2}:]
 By assumption there are $a_1, a_2\in L$ with $(a_1,a_2)\ne (0,0)$ and
$\sigma (a_1)a_1+\sigma (a_2)a_2 =0$. Set $v:= a_1e_1+a_2e_2$. We have $v\ne 0$ and $\langle v,v\rangle =0$. Hence $e_2$ and $v$ are not proportional and so they form a basis of $L^2$. Since $\langle \ ,\ \rangle$ is non-degenerate and  $\langle v,v\rangle =0$, we have  $\langle v,e_2\rangle \ne 0$. Taking a multiple of $v$ if necessary
we reduce to the case $\langle v,e_2\rangle =1$. Thus $\langle e_2,v\rangle =1$. Take $M\in M_{n,n}(L)$ defined by $Mv=cv$ and $Me_2 = \mu v+ce_2$. Apply Proposition \ref{i2.1}.\end{proof}

\begin{proof}[Proof of Theorem \ref{a3}:] Taking $M-c\II _{2,2}$ instead of $M$ and applying Remark \ref{a4} we reduce to the case $c=0$.

Set $\delta := \langle v,v\rangle$. Write $v =a_1e_1+a_2e_2$ and set $w:= -\sigma (a_2)e_1+\sigma (a_1)e_2$. We have $\langle v,w\rangle = -\sigma (a_1)\sigma (a_2)
+\sigma (a_2)\sigma (a_1)=0$ and hence $\langle w,v\rangle =0$. Since $\delta =\langle v,v\rangle =\sigma (a_1)a_1+\sigma (a_2)a_2$ and $\sigma (\sigma (a_i))=a_i$,
we have $\langle w,w\rangle =\delta$. Since $\delta \ne 0$ and $\langle v,w\rangle =0$, $v,w$ are a basis of $L^2$. Write $M =(b_{ij})$, $i, j=1,2$, with respect to the basis
$v,w$. By assumption we have $b_{11}=b_{21}=0$. Since $M$ has a unique eigenvalue, we have $b_{22}=0$. Assume $M\ne 0\II _{2,2}$ and hence $b_{12}\ne 0$.
Set $\mu := b_{12}$. 

\quad (i) If $\delta \in \Delta$, then $\delta \in \Nm (M)$ by Remark \ref{i0.1.0}. Now assume $0\in \Nm (M)$, i.e. assume the existence of $u=xv+yw$ such that $\langle u,u\rangle =1$ (i.e. such that $\sigma (x)x+\sigma (y)y =1/\delta$)
and $\langle u,Mu\rangle =0$ (i.e. $0 = \langle xv+yw,\mu yv\rangle = \mu \delta \sigma (x)y$).  Hence either $y=0$ or $x=0$. Hence $u$ is either a multiple of $v$ or a multiple of $w$. Since $\langle v,v\rangle =\langle w,w\rangle =\delta$, Lemma \ref{i0.1} gives $\delta \in \Delta$.

\quad (i) Now assume $\delta \in \hat{\Delta}$. Since $\hat{\Delta}$ is a multiplicative group, there
is $t\in L$ such $\sigma (t)t =1/\delta$. Set $v_1:= tv$ and $v_2:= tw$. We have $\langle v_i,v_i\rangle =1$ and $\langle v_i,v_j\rangle =0$ if $i\ne j$. 
Let $(m_{ij})$ be the matrix associated to $M$ in the basis $v_1,v_2$. We have $m_{11}=m_{12}=0$. Since $M$ has a unique eigenvalue, we have
$m_{22} =0$. Assume $M\ne 0\II _{2\times 2}$, i.e. assume $m_{12}\ne 0$. Taking $(1/m_{12})M$ instead of $M$ and applying
Remark \ref{a4} we reduce to the case $m_{12}=1$. Take $u=xw_1+yw_2$ with $\langle u,u\rangle =1$, i.e. with $\sigma(x)x+\sigma (y)y =1$. Set $\gamma := \sigma(x)x$
and hence $\sigma (y)y =1-\gamma$. Since $0\in \Nm (M)$, to check all other elements of $\Nm (M)$ we may assume $xy\ne 0$, i.e. $\gamma \notin \{0,1\}$. Note that $\gamma$
is an arbitrary element of  $\hat{\Delta} \cap (1-\hat{\Delta })$.  We fix $\gamma$, but we only take, $x, y$ with $\sigma (x)x=\gamma$ and $\sigma (y)y =1-\gamma$.  
We have $\langle u,Mu\rangle = \langle xw_1+yw_2,yw_1\rangle = \sigma (x)y$. Note that $\sigma (x)y\cdot \sigma (\sigma (x)y) = \gamma (1-\gamma)$.
Fix $w\in L$ such that $\sigma (w)w =\gamma (1-\gamma )$. Take any $x$ with $\sigma (x)x=\gamma$. Note that $x\ne 0$, because $\gamma \ne 0$. Taket $y:= w/\sigma (x)$.
To conclude the proof of part (b) it is sufficient to prove that $\sigma (y)y=1-\gamma$. We have $\sigma (y)y = \sigma (w)w/x\sigma (x) = \gamma (1-\gamma )/\gamma$.\end{proof}

Lemma \ref{a3.0} and Proposition \ref{a3.1} are, respectively, the analogue of \cite[Lemma 3.6 and Theorem 1.2 (d)]{cjklr}. In their statements the last $\bigcup$ is a union of circles with center $0$, in which if we take $d\in \hat{\Delta}\cap (1-\hat{\Delta})$ as a parameter space the circles coming from $d$ and $1-d$ are the same and we do not claim that $\bigcup$ is a disjoint union (see \cite[Example 3.7]{cjklr}).
\begin{lemma}\label{a3.0}
Fix $b\in L^\ast$ and let $M =(m_{ij})$ be the $2\times 2$ matrix with $a_{11}=1$, $a_{21}=a_{22}=0$ and $a_{12}=b$. Then $\Nm (M) =\{0,1\}\cup \bigcup_{d(1-d),d\in \hat{\Delta}\cap (1-\hat{\Delta})}  C(\sigma(b)bd(1-d),d)$.
\end{lemma}

\begin{proof}
The vector $e_1$ gives $1\in \Nm (M)$. The vector $e_2$ gives $0\in \Nm (M)$. Take $u=xe_1+xe_2\in L^2$ such that $\langle u,u\rangle =1$, i.e. such that $\sigma (x)x +\sigma (y)y=1$. Set $d:= \sigma (x)x \in \Delta$. We have $\sigma (y)y=1-d$
and hence $d\in (1-\Delta )$. We have $\langle u,Mu\rangle =\langle xe_1+y_2,(x+by)e_1\rangle =d+b\sigma (x)y$. Set $m:= b\sigma (x)y$. We have $\sigma (m)m = \sigma (b)b
\sigma (y)y\sigma (x)x = \sigma (b)bd(1-d)$. Thus $m \in C(bd(1-d),0)$ and hence $\langle u,Mu\rangle \in d + 
C(\sigma (b)bd(1-d),0)$. Since $0, 1\in \Nm (M)$, from now on we may assume $d\notin \{0,1\}$ and prove the other inclusion. Take any $m'\in C(\sigma (b)bd(1-d),0)$, i.e. with $\sigma (m')m' = \sigma (b)bd(1-d)$. Take $x_1\in C(d,0)$. Since $d\ne 0$, we have $x_1\ne 0$. Set $y_1:= m'/(b\sigma (b)x_1$ and $u':= x_1e_1+y_1e_2$. We have $y_1\in C((1-d),0)$, $\langle u',Mu'\rangle = d+m'$ and $\sigma (x_1)x_1+\sigma (x_2)x_2 = \sigma (b)bd(1-d)$. 
\end{proof}

\begin{proposition}\label{a3.1}
Assume $n=2$ and that $M$ has two eigenvalues $c_1,c_2\in L$, $c_1\ne c_2$,  with eigenvectors $v_i$ for $c_i$ with $\langle v_1,v_1\rangle  \in \hat{\Delta}$.

\quad (i) If $\langle v_1,v_2\rangle =0$, then $M$ is unitarily equivalent to $c_1\II _{1\times 1}\oplus c_2\II _{1\times 1}$.

\quad (ii) Assume $\langle v_1,v_2\rangle \ne 0$. Then there is $\mu \in L^\ast$ such that $\Nm (M)$ is the union of $\{c_1,c_2\}$ and a union of circles  $C(\sigma (\mu)\mu d(1-d),d)$ with $d\in \hat{\Delta}\cap (1-\hat{\Delta})$. 
\end{proposition}

\begin{proof}
Set $\delta := \langle v_1,v_1\rangle$. Since $1/\delta \in \hat{\Delta}$ (Remark \ref{w2}), there is $t\in L^\ast$ such that $\sigma (t)t = 1/\delta$.
Set $w_1:= tv_1$. We have $\langle w_1,w_1\rangle =1$ and $Mw_1=c_1w_1$. Write $w_1= a_1e_1 +a_2e_2$ for some $a_1,a_2\in L$.
Set $w_2:= -\sigma (a_2)e_1+\sigma (a_1)e_1$. Note that $\langle w_2,w_1\rangle =0$ and $\langle w_2,w_2\rangle =1$. Taking $M-c_2\II _{2,2}$
instead of $M$ and applying Remark \ref{a4} we reduce to the case $c_2=0$ and hence $c:= c_1-c_2\ne 0$. Taking $(1/c)M$ instead of $M$
and applying Remark \ref{a4} we reduce to the case $c =c_1-c_2=1$. Write $M= (m_{ij})$, $i=1,2$, in the orthonormal basis $w_1$ and $w_2$.
We have $m_{11}=1$ and $m_{21}=m_{22}=0$. If $m_{12} =0$, then $M$ is unitary equivalent to the matrix $c_1\II _{1\times 1}\oplus c_2\II _{1\times 1}$.
We have $m_{12}=0$ if and only if $w_2$ is proportional to $v_2$, i.e. (being $v_2$ linearly independent from $tv_1=w_1$) if and only if  $\langle v_1,v_2\rangle =0$.
Apply Lemma \ref{a3.2} with $\mu := m_{12}/(c_1-c_2)$.
\end{proof}

\begin{proposition}\label{a3.2}
Take $M\in M_{2,2}(L)$ with eigenvalues $c_1,c_2\in L$, $c_1\ne c_2$, and take $v_i\in L^2$ such that $Mv_i=c_iv_i$ and $v_i\ne 0$. Assume
$\langle v_i,v_i\rangle =0$ for all $i$. Set $\Dd := \{t\in L\mid t+\sigma (t) =1\}$. Then $\Nm (M) =\Dd$
\end{proposition}

\begin{proof}
Taking $(1(c_2-c_1)(M-c_1\II_{2\times 2})$ and applying Remark \ref{a4} we reduce to the case $c_1=0$ and $c_2=1$. Since $\langle \ ,\ \rangle$ is non-degenerate and $\langle v_i,v_i\rangle =0$ for all $i$, we have $\langle v_1,v_2\rangle \ne 0$. Taking $(1/\langle v_1,v_2\rangle  )v_2$
instead of $v_2$ we reduce to the case $\langle v_1,v_2\rangle =1$. Thus $\langle v_2,v_1\rangle =1$. Take $u=xv_1+yv_2$ with $\langle u,u\rangle=1$, i.e.
with $\sigma (x)y+\sigma (y)x =1$. Note that $x\ne 0$ and $y\ne 0$. We have $\langle u,Mu\rangle = \langle xv_1+yv_2,yv_1 \rangle= \sigma (x)y$.  For any
$b\in L^\ast$ let $N(b)$ be the set of all $\sigma (x)b$ with $\sigma (x)b+\sigma (b)x =1$. Fix $y\in L^\ast$. Set $t:= x/y$. Note that $t\ne 0$. Since $\sigma (x)y+\sigma (y)x =1$, we have $\sigma (t)y\sigma (y) +ty\sigma (y)  =1$and $\langle u,Mu\rangle =\sigma (t)y\sigma (y)$.Take $t_0,y_0,y\in L^\ast$
such that $t_0\sigma (t_0)y_0\sigma (y_0) +t_0y_0\sigma (y_0)  =1$ and set $t_1:= t_0y_0\sigma (y_0)/(y\sigma (y))$. We have $\sigma (t_1)y\sigma (y) + t_1y\sigma (y)  = t_0\sigma (t_0)\sigma (t_0) + t_0y_0\sigma (y_0)$. Hence $N(y)\supseteq N(y_0)$. By symmetry we get $N(y) =N(y_0)$. Taking $y_0=1$ we get
$\Nm (M) =N(1)$. Note that  $t+\sigma (t) =1$ if and only if $\sigma (t)+t=1$. Hence $\sigma (N(1)) =N(1)$. Thus $\Nm (M) =\Dd$. 
\end{proof}

\begin{proposition}\label{a0.1}
Fix $n>2$ and assume $\Delta _n=\Delta$. Take $M\in M_{n,n}(L)$ with an eigenvalue $c\in L$ with an eigenspace of dimension $n-1$. Then 
one of the following cases occurs:
\begin{enumerate}
\item $M$ is unitarily equivalent to $c\II _{n-1,n-1}\oplus d\II _{1\times 1}$ (unitary direct sum) for some $d\ne c$;
\item $M$ is unitarily equivalent to $c\II _{n-2,n-2}\oplus M'$ (unitary direct sum) with $M'$ as in case (ii) of Proposition \ref{a3.1};
\item $M$ is unitarily equivalent to $c\II _{n-2,n-2}\oplus M'$ (unitary direct sum) with $M'$ as in Theorem \ref{a3};
\item $M$ is unitarily equivalent to $c\II _{n-2,n-2}\oplus M'$ (unitary direct sum) with $M'$ as in Proposition \ref{a3.2};
\item $M$ is unitarily equivalent to $c\II _{n-2,n-2}\oplus M'$ (unitary direct sum) with $M'$ as in Proposition \ref{i2.1}.
\end{enumerate}
\end{proposition}

\begin{proof}
Let $W\subset L^n$ be the $c$-eigenspace space of $M$. Taking $M-c\II_{n,n}$ instead of $M$ and applying Remark \ref{a4} we reduce to the case $c=0$. Hence $Mw=0$ for all $w\in W$. By Lemma \ref{w0.1} there are $f_1,\dots ,f_{n-2}\in W$ such that $\langle f_i,f_i\rangle =1$ for all $i$ and $\langle f_i,f_j\rangle =0
$ for all $i\ne j$. Let $V$ be the linear span of $f_1,\dots ,f_{n-2}$. Set $V^\bot := \{x\in L^n\mid \langle w,x\rangle =0$ for all $w\in W\}$. By the choice of $f_1,\dots ,f_{n-2}$ we have $V\cap V^\bot =\{0\}$. Since $\langle \ ,\ \rangle$ is non-degenerate, we have $\dim (V) +\dim (V^\bot )=n$ and so $L^n = V\oplus V^\bot$ (unitary direct sum). Fix
$w, v\in V$. We have $\langle v,M^\dagger w \rangle= \langle Mv,w\rangle =\langle 0,w\rangle =0$. Since the restriction of $\langle \ ,\ \rangle$ to $V$ is non-degenerate, we get
$M^\dagger w =0$. Hence $M^\dagger w=0$ for all $w\in W$. Fix $m\in V^\bot$  and $v\in W$. We have
$\langle v,Mm\rangle = \langle M^\dagger v,m\rangle =\langle 0,m\rangle =0$. Since this is true for all $v\in W$, we get $Mm \in V^\bot$.
Hence $MV^\bot \subseteq V^\bot$. Set $B:= M_{|V^\bot}$, seen as a map $V^\bot \to V^\bot$. All the eigenvalues of $M$ are in $L$ and we call $d$ the other eigenvalue. The matrix $B$ has eigenvalues $0$ and $d$ with $0$ the eigenspace contains
$u\in W\cap V^\bot$, $u\ne 0$. 

\quad (a) First assume $d\ne 0$ and hence there is $v\in V^\bot$ with $Mv =dv$ and $v \ne 0$. We have $\langle z,z\rangle \in \Delta _n =\Delta$ for all $z\in V^\bot$.
Hence if either $\langle u,u\rangle \ne 0$ or $\langle v,v\rangle \ne 0$, then we apply Proposition \ref{a0.1} and get that we are either in case (1) or case (2).
If $\langle u,u\rangle =\langle v,v\rangle =0$, then we apply Proposition \ref{a3.2}.

\quad (b) Now assume $d=0$. By assumption $Lu$ is the only one-dimensional subspace of $V^\bot$ sent into itself by $B$. If $\langle u,u\rangle =0$, then 
we apply Proposition \ref{i2.1}.
If $\langle u,u\rangle \ne 0$, then we apply Theorem \ref{a3}.\end{proof}

\section{Matrices with coefficients in $K$}\label{Scc}

Take $M =(m_{ij})\in M_{n,n}(K)$. The set $C_n(1,K)$ is the set of all solutions $(x_1,\dots ,x_n)\in K^n$  of the equation
\begin{equation}\label{eqcc1}
x_1^2+\cdots +x_n^2 =1
\end{equation}
Thus $\Nm (M)_K$ is the set of all 
\begin{equation}\label{eqcc2}
\sum _{i,j} m_{ij}x_ix_j
\end{equation}
with $x_1,\dots ,x_n$ satisfying (\ref{eqcc1}).

\begin{lemma}\label{cc4}
Take $M =(m_{ij})$, $N=(n_{ij})\in M_{n,n}(K)$. 
\begin{enumerate}
\item If $m_{ii} =n_{ii}$ for all $i$ and $m_{ij}+m_{ji} = n_{ij}+n_{ji}$ for all $i\ne j$, then $\Nm (M)_K=\Nm (N)_K$.
\item We have $\Nm (B)_K= \Nm (M)_K$ for the matrix $B:= (b_{ij})$ with $b_{ii}=m_{ii}$ for all $i$, $b_{ij}=0$ for all $i<j$ and $b_{ij} = m_{ij}+m_{ji}$ for all $i>j$.
\item If $\mathrm{char}(K) \ne 2$, the matrix $A:= (a_{ij})$ with $a_{ij}=(m_{ij}+m_{ji})/2$ for all $i, j$ is symmetric and $\Nm (A)_K =\Nm (M)_K$.
\end{enumerate}
\end{lemma}

\begin{proof}
The equation (\ref{eqcc2}) is the same for $M$ (i.e. with $m_{ij}$ as coefficients) and for $N$ (i.e. with $n_{ij}$ as coefficients). The last two assertions of Lemma \ref{cc4} follows from the first one.
\end{proof}

\begin{remark}\label{cc2}
For all $c, d\in K$ and all $M\in M_{n,n}(K)$ we have $\Nm (cM+d\II _{n,n})_K= c\Nm (M)_K+d$.
\end{remark}

\begin{remark}\label{cc3}
The vectors $e_1,\dots ,e_n$ prove that for any $M\in M_{n,n}(K)$ the diagonal elements of $M$ are contained in $\Nm (M)_K$.
\end{remark}

\begin{proof}[Proof of Proposition \ref{cc1}:]
Write $M =(m_{ij})$. Taking $M-m_{11}\II _{n,n}$ instead of $M$ we reduce to the case $m_{11}=0$ by Remark \ref{cc2}. If $M$ is antisymmetric and $m_{ii}=0$ for all $i$, then $\Nm (M)_K =\Nm (0\II _{n\times n})_K=\{0\}$ by Lemma \ref{cc4}.

Now assume $\sharp (\Nm (M)_K) =1$. Since the diagonal elements of $M$ are contained in $\Nm (M)_K$ by Remark \ref{cc3},
we have $m_{ii}=0$ for all $i$. Assume $m_{ij}\ne 0$ for some $i\ne j$. The first part of the proof of Proposition \ref{i1} with $e_i,e_j$ instead of $e_1,e_2$ (i.e. the
part with $x,y\in K$)
gives $m_{ji} =-m_{ij}$.
\end{proof}

\begin{remark}\label{cc5}
Assume $\mbox{char}(K) =2$. Then $x_1^2+\cdots +x_n^2 = (x_1+\cdots +x_n)^2$. Hence the elements of (\ref{eqcc2}) coming from the solutions of (\ref{eqcc1})
are the ones coming from the solutions of
\begin{equation}\label{eqcc3}
x_1+\cdots +x_n=1
\end{equation}
Substituting $x_n =1+x_1+\cdots +x_{n-1}$ in (\ref{eqcc2}) we get that $\Nm (M)_K$ is the image of a map $f_M: K^{n-1}\to K$ with $f_M$ a polynomial in $x_1,\dots ,x_{n-1}$
with $\deg (f_M)\le 2$. If $\deg (f_M) =1$, then $f_M$ is surjective, i.e. $\Nm (M)_K =K$. If $\deg (f_M) =0$, then $f_M$ is constant
and hence $\sharp (\Nm (M)_K)=1$. Let $g_M$ be the homogeneous degree $2$ part of $f_M$ and let $A =(a_{ij})$,
$i,j=1,\dots ,n-1$, be the matrix associated to $g_M$ with $a_{ij}= 0$ if $i<j$. We have $a_{ii} = m_{ii} +m_{nn}$ and $a_{ij} = m_{ij}+m_{ji}$ for all $i\ne j$ with $i,j<n$. Since
$\mathrm{char}(K) =2$, we have $a_{ii}=0$ if and only if $m_{ii}=m_{nn}$ and $a_{ij}=0$ (with $i\ne j$) if and only if $m_{ij}=m_{ji}$. Thus $g_M =0$
if and only if all diagonal elements of $M$ are the same and the top $(n-1)\times (n-1)$ principal submatrix of $M$ is symmetric.

\quad (a) Assume $g_M\ne 0$ and that $K$ is infinite. 

\quad {\emph {Claim 1:}} If $a_{ij}\ne 0$ for some $i\ne j$, then $\Nm (M)_K =K$.

\quad {\emph{Proof of Claim 1:}} Up to a permutation of $e_1,\dots ,e_{n-1}$ we may assume $a_{12} \ne 0$. 
We have $f_M(x_1,x_2,0,\dots ,0) = (m_{11} +m_{nn})x_1^2 +a_{12}x_1x_2 + (m_{22}+m_{nn})x_2^2+\beta x_1+\gamma x_2+\delta$ for some
$\beta, \gamma ,\delta \in K$. Hence it is sufficient to prove that the image of the map $\psi : K^2\to K$ induced by the polynomial $f_M(x_1,x_2,0,\dots ,0)$
has the cardinality of $K$. We will prove that $\psi$ is surjective. Take $b\in K$ such that $a_{11}b \ne -\gamma$. The polynomial $f_M(b,x_2,0\dots ,0)$
is a non-constant degree $1$ polynomial and hence it induces a surjection $K\to K$. Thus $\psi$ is surjective.
Now assume $a_{ii}\ne 0$ for some $i$, i.e. $m_{ii}\ne m_{nn}$ for some $i<n$.

\quad{\emph{Claim 2:}} Assume $a_{ij}= 0$ for all $i\ne j$, but $a_{ii} \ne 0$ for some $i<n-1$. If $K$ is infinite, then $\Nm (M)_K$ has the cardinality of $K$.

\quad {\emph{Proof of Claim 2:}} We have $g_M(x_1,\dots ,x_{n-1}) =\sum _{i=1}^{n-1} a_{ii}x_i^2 \ne 0$ and
$$f_M(x_1,\dots ,x_{n-1}) = g_M(x_1,\dots ,x_{n-1})+\ell (x_1,\dots ,x_{n-1}) + \gamma$$ for some $\gamma \in K$ and a linear form $\ell \in K[x_1,\dots ,x_{n-1}]$. Up to a permutation of the indices we may assume that
$a_{11}\ne 0$. Fix any $(b_2,\dots ,b_{n-1}) \in K^{n-2}$ and call $\phi : K\to K$ the map induced by $f_M(x_1,b_2,\dots ,b_{n-1})$. Since $f_M(x_1,b_2,\dots ,b_{n-1})$ is a degree $2$ non-constant polynomial, each fiber of $\phi$ has at most cardinality $2$. Hence $\phi (K)$ and $K$ have the same cardinality.

\quad (b) Assume $g_M \equiv 0$. In particular $m _{ii}=m_{nn}$ for all $i<n$. We have $f_M(0,\dots ,0) =a_{nn}$ and $f_M(x_1,\dots ,x_m) =b_1x_1+\cdots +b_{n-1}x_{n-1}
+a_{nn}$ with $b_i = m_{ni} +m_{in}$. Hence $f_M$ is surjective if and only if $m_{ni}\ne m_{in}$ for some $i<n$, while $f_M$ is constant if $m_{ni}=m_{in}$ for all $i<n$.
\end{remark}

\begin{proof}[Proof of Proposition \ref{cc6}:]
The proposition was proved in Remark \ref{cc5}, with as a bonus the discussion of some cases with $\Nm (M)_K=K$.\end{proof}

\providecommand{\bysame}{\leavevmode\hbox to3em{\hrulefill}\thinspace}

\end{document}